\documentclass[11pt, reqno]{amsart}
\usepackage{amsmath,amsthm,amssymb,amsfonts}
\usepackage{epsfig}

\theoremstyle{plain}
\newtheorem{thm}{Theorem}[section]
\newtheorem{prop}[thm]{Proposition}
\newtheorem{lem}[thm]{Lemma}
\newtheorem{question}[thm]{Question}
\newtheorem{cor}[thm]{Corollary}
\newtheorem{conj}[thm]{Conjecture}

\theoremstyle{definition}

\newtheorem{rem}[thm]{Remark}
\newtheorem{defn}[thm]{Definition}
\newtheorem{eg}[thm]{Example}
\newtheorem{subtitle}[thm]{}
\newtheorem{ex}{Exercise}[section]
\numberwithin{equation}{section}

\def\a{\alpha}

\def\e{\epsilon}

\def\l{\lambda}

\def\n{\,\vert\,}

\def\w{\omega}

\def\ms{\medskip}

\def\ni{\noindent}
\def\ti{\tilde}

\def\Hess{{\rm Hess\/}}
\def\sgn{{\rm sgn\/}}

\def\ind{{\rm ind\/}}

\def\R{\mathbb{R} }
\def\C{\mathbb{C}}
\def\P{\mathbb{P}}

\def\N{\mathbb{N}}
\def\Z{\mathbb{Z}}

\newcommand{\beq}{\begin{equation}}
\newcommand{\eeq}{\end{equation}}
\newcommand{\beg}{\begin{eg}}
\newcommand{\eeg}{\end{eg}}
\newcommand{\bthm}{\begin{thm}}
\newcommand{\ethm}{\end{thm}}
\newcommand{\bprop}{\begin{prop}}
\newcommand{\eprop}{\end{prop}}
\newcommand{\bcor}{\begin{cor}}
\newcommand{\ecor}{\end{cor}}
\newcommand{\blem}{\begin{lem}}
\newcommand{\elem}{\end{lem}}
\newcommand{\bca}{\begin{cases}}
\newcommand{\eca}{\end{cases}}
\newcommand{\brem}{\begin{rem}}
\newcommand{\erem}{\end{rem}}
\newcommand{\bconj}{\begin{conj}}
\newcommand{\econj}{\end{conj}}
\newcommand{\bpm}{\begin{pmatrix}}
\newcommand{\epm}{\end{pmatrix}}
\newcommand{\bbm}{\begin{bmatrix}}
\newcommand{\ebm}{\end{bmatrix}}
\newcommand{\bvm}{\begin{vmatrix}}
\newcommand{\evm}{\end{vmatrix}}
\newcommand{\bdefn}{\begin{defn}}
\newcommand{\edefn}{\end{defn}}
\newcommand{\bsub}{\begin{subtitle}}
\newcommand{\esub}{\end{subtitle}}
\newcommand{\bex}{\begin{ex}}
\newcommand{\eex}{\end{ex}}
\newcommand{\ben}{\begin{enumerate}}
\newcommand{\een}{\end{enumerate}}

\begin{document}

\title{Non-orientable Lagrangian surfaces in rational $4-$manifolds}
\author{Bo Dai}
\address{School of Mathematical Sciences\\
Peking University\\ Beijing 100871, China}
\email{daibo@math.pku.edu.cn}
\author{Chung-I Ho}
\address{Department of Mathematics\\
National Kaohsiung Normal University\\ Kaohsiung 82446, Taiwan}
\email{ciho@nknu.edu.tw}
\author{Tian-Jun Li}
\address{Department of Mathematics \\
University of Minnesota\\ Minneapolis, MN 55455}
\email{tjli@math.umn.edu}

\ms


\begin{abstract} 
We show that for any 
 {\bf nonzero} class $A$  in $H_2(X;\Z_2)$ in a rational $4-$manifold $X$,  $A$  is represented by a non-orientable embedded Lagrangian surface 
$L$ (for some symplectic structure) if and only if 
$\mathcal P(A)\equiv \chi(L) \pmod4,$
where $\mathcal P(A)$ denotes the mod 4 valued Pontrjagin  square of $A$. 
\end{abstract}

\maketitle
\tableofcontents
\section{Introduction}


A smooth immersion or embedding $f:L\rightarrow X$ from a smooth manifold $L$ into a symplectic manifold $(X,\omega)$ is called Lagrangian if $\dim L=\frac12\dim X$ and $f^*\omega=0$.
The existence of Lagrangian submanifolds is an important problem in symplectic topology and is studied by many people, see \cite{A1, ALP,B, LS, P} and others.
When $X$ is a uniruled 4-manifold, it is well-known that the only Lagrangian embedding of closed orientable surfaces to $X$ are spheres and tori, and Lagrangian tori are homologically trivial in $X$. 
The existence question for Lagrangian spheres in a uniruled 4-manifold was completely answered in \cite{LW}.

 The  focus of this paper will be on the existence of embedded non-orientable Lagrangian surfaces for a
 given   mod 2 homology class.
First of all, we have the following simple observation by the Lagrangian immersion h-principle and the Lagrangian surgery construction.

\begin{prop}\label{general existence}
Let $(X, \omega)$ be a  symplectic 4-manifold.
For any mod 2 homology class $A$, there exists an embedded non-orientable Lagrangian surface representing this class. 
\end{prop}

In light of the general existence, we ask the following question. 
\begin{question} Given a mod 2 class $A$, what are the possible topological type of non-orientable Lagrangian surfaces in the class $A$? Especially, 
what is the maximal Euler number (equivalently, the minimal genus)?
\end{question}

We will study this question for rational  4-manifolds. Here, a smooth $4-$manifold is called rational if it is $S^2\times S^2$ or $\C\P^2\# k\overline{\C\P^2}, k\in \Z_{\geq 0}$.

Recall that each non-orientable surface is diffeomorphic to $N_k=k\R\P^2=\R\P^2\#\cdots\#\R\P^2$ for some $k\in\N$. 
The Euler number of $N_k$ is $2-k$,
and the genus of $N_k$ is defined to be $k$. 
Audin's congruence theorem (\cite{A1})  states that, for a mod 2 class $A$ in an arbitrary symplectic 4-manifold $(X, \omega)$, 
if $A$ is represented by an embedded non-orientable Lagrangian surface $L$,   then  the Pontrjagin square of $A$ is congruent to the Euler number of $L$  modulo $4$.

The Pontrjagin square referred here  is a certain  cohomology operation
$\mathcal P:H^2(X;\Z_2)\to H^4(X;\Z_4)$ which is a lift of the mod 2 cup product (cf. eg. \cite{T}). 
Furthermore, if $A$ is the reduction of an integral class $\bar A$, then $\mathcal P(A)$ is the mod 4 reduction of $\bar A^2$. 
In particular, if $X$ is a rational manifold, then  $H_2(X;\Z)$ has no torsions, especially $2$-torsions, thus    every mod 2 class $A$ has an
integral lift $\bar A$, and  $\mathcal P(A) \equiv \bar A^2\pmod 4$.

 For the zero class, it follows from Givental's construction in $\R^4$  (\cite{Gi})
 that there are non-orientable  Lagrangian surfaces with  Euler number divisible by $4$, except possibly the Klein bottle.  Remarkably, it was shown by  Shevchishin and Nemirovski (\cite{S} and \cite{N}) that  the mod 2 class of a  Lagrangian Klein Bottle in a uniruled manifold must be nonzero. 
 Together with Audin's congruence theorem,  the problem for the zero class is completely understood
 for a uniruled manifold.

 Now we assume that $A$ is a nonzero class. 
The first step is to consider the situation where  the symplectic form $\omega$  is not fixed.  
  We apply 
the classification of Lagrangian spheres in \cite{LW}  together with the Lagrangian blow-up construction in \cite{Ri}, and Givental's local construction  
in \cite{Gi} to   show that Audin's congruence is also sufficient  when $A$ is a nonzero class in a rational 
manifold.

\begin{thm}\label{main}
Let  $X$ be a rational 4-manifold
and $A$ be a nonzero class in $H_2(X;\Z_2)$.  Let $\mathcal P(A)$ denote the Pontrjagin  square of $A$. Then $A$ is represented by an embedded non-orientable Lagrangian surface of Euler number $\chi$ for some symplectic structure if and only if 
$$\mathcal P(A)\equiv \chi \pmod4.$$
\end{thm}

\begin{rem} 
If we  denote $|\mathcal P(A)|$ the normalized $\mathcal P(A)$ taking values in $\{-2,-1, 0, 1\}$, then  the minimal genus of embedded non-orientable Lagrangian surfaces in a class $A$ is $$2-|\mathcal P(A) | \in \{1, 2, 3, 4\}.$$ 

\end{rem}

Notice that  for Lagrangian immersions,  Proposition \ref{general existence}  holds for {\bf any symplectic structure}.
The next step is to fix the symplectic form, or equivalently,  classify for which symplectic forms there exist an embedded Lagrangian surface representing $A$.     
A distinct feature is that, unlike for a Lagrangian sphere which only exists for a codimension one locus of the symplectic cone due to the null symplectic area condition, 
this seems to be  an open condition.
We will deal with this problem in a future work.

The structure of this paper is as following. In Section 2, we introduce several general approaches in constructing Lagrangian submanifolds and use them to prove Proposition \ref{general existence}.
In Section 3, we construct embedded Lagrangian surfaces with desired genus and prove Theorem \ref{main}. 

\ms
\ni {\bf Acknowledgement}. The third named author would like to thank  Banghe Li for   useful discussions on Proposition \ref{general existence}. The research of first named author is partially supported by NSFC 11771232 and 11431001. The research of second named author is partially supported by MOST 105-2115-M-017-005-MY2. The research of third named author is partially supported by NSF. 

\section{Constructing non-orientable Lagrangian surfaces}

\subsection{Existence of immersed Lagrangian surfaces}

In this subsection we establish  the existence of immersed Lagrangian surfaces in an arbitrary symplectic 4-manifold.


Let us recall Gromov and Lee's h-principle (\cite{Gr}, \cite{Lee}). 
Let $L$ be a closed $n$-manifold and $(W, \omega)$ be a symplectic $2n$-manifold. 

A smooth map $f:L\to (W, \omega)$ is called an {\it almost (or formal) Lagrangian immersion\/}
if  the following two conditions are satisfied:

 (1) $f^*[\omega]=0$ in $H^2(L;\mathbb R)$.

(2)  there is an injective bundle map  $F: TL\to f^* TW$ over $L$
such that $F(T_pL)\subset (f^*TW|_p,  f^*\omega |_p)$ is a Lagrangian subspace for any $p\in L$. 

\begin{thm}\label{h-principle}  (Gromov, Lee) Every almost Lagrangian immersion is homotopic through almost Lagrangian immersions to a Lagrangian immersion.
\end{thm}

To apply this h-principle, here is a useful observation. If we take an $\omega$-compatible almost complex structure on $W$, then symplectic vector bundles and complex vector bundles are closely related and Lagrangian subbundles correspond to real subbundles. 
Hence condition (2) can be replaced by an isomorphism as complex vector bundles
\begin{equation} \label{bundle condition}
TL\otimes \mathbb C \cong f^* TW
\end{equation}
In our situation, we have
\begin{lem}\label{classify bundle}
Rank two complex vector bundles over a non-orientable surface are classified by $w_2$. 

\end{lem}

\begin{proof}
Let $\Sigma$ be a non-orientable surface and $E$  a rank two complex vector bundle over $\Sigma$. 
In particular, there is an almost complex structure $J$ on $E$.
For dimensional reason, there is a nowhere zero section $\tau\in \Gamma(E)$.
$\tau\oplus J\tau$ forms a trivial complex line bundle and induces a splitting $\mathbb C \oplus \xi$. 
The complex line bundle $\xi$ is classified by $c_1\in H^2(\Sigma; \mathbb Z)=\mathbb Z_2$. 
Notice that $c_1\equiv w_2$ under the mod 2 reduction homomorphism
$H^2(\Sigma; \mathbb Z)\to H^2(\Sigma; \mathbb Z_2)$, which is an isomorphism in this case.
\end{proof}


\begin{prop} \label{h principle w2 version}
Let $(X, \omega)$ be a symplectic 4-manifold and $A$ a mod 2 homology class. Suppose 
$\Sigma$ is a non-orientable surface and $f:\Sigma\to X$ is a smooth map such
that $f_*([\Sigma])=A$ and $\chi(\Sigma)\equiv \langle w_2(X), A\rangle   \pmod 2$. Then there is a Lagrangian immersion from $\Sigma$ to $(X, \omega)$ which is homotopic to $f$. 
\end{prop}
\begin{proof}
By Theorem \ref{h-principle}  it suffices to show that $f$ is an  almost Lagrangian immersion. 
Since $\Sigma$ is a non-orientable surface,  $H^2(\Sigma;\mathbb R)=0$ and hence $f^*[\omega]$ is automatically zero.

Let us now analyze the bundle $f^*TX$ as a real vector bundle by calculating the Stiefel-Whitney classes. 
Firstly, $w_1(f^*TX)=f^* w_1(TX)=0$ since $X$ is orientable. 
Since  $\Sigma$ is  non-orientable, we have the pairing $\langle w_2(f^*TX), [\Sigma]\rangle$, and 
$$\langle w_2(f^*TX), [\Sigma]\rangle= \langle w_2(TX), A\rangle.  $$
For the bundle  $T\Sigma\otimes \mathbb C$, as a real bundle, 
$$T\Sigma\otimes \mathbb C=T\Sigma\oplus T\Sigma$$
has  $w_1=0$ and, by Wu's formula (cf. Theorem 11.14 in \cite{MS74}),
$$w_2=2w_2(T\Sigma)+w_1(T\Sigma)\cdot w_1(T\Sigma)=w_1(T\Sigma)\cdot w_1(T\Sigma)=w_2(T\Sigma).$$
Since $w_2(T\Sigma)$ is the mod $2$ reduction of the Euler class of $T\Sigma\to \Sigma$ we have 
 $$\langle w_2(T\Sigma\otimes \mathbb C), [\Sigma]\rangle \equiv \chi(\Sigma) \pmod 2.  $$
It follows from Lemma \ref{classify bundle} that  \eqref{bundle condition} holds if and only if 
$$\chi(\Sigma)\equiv \langle w_2(TX), A\rangle \pmod 2 .$$
\end{proof}

\begin{rem}
Since  $\langle w_2(W), A\rangle \equiv A\cdot A \pmod 2$, 
the condition
$\chi(\Sigma)\equiv \langle w_2(TX), A\rangle \pmod 2$
is equivalent to   
$\chi(\Sigma)\equiv A\cdot A \pmod 2$.  
\end{rem}







\subsection{Existence of embedded non-orientable Lagrangian surfaces}

In this subsection, we review the Lagrangian surgery and Givental's beautiful local constructions. 
Then we use these tools to construct new embedded Lagrangian surfaces from old or immersed ones in a given symplectic 4-manifold.

\subsubsection{Lagrangian surgery}(\cite{P})
A Lagrangian surgery is a desingularization of a transversal Lagrangian intersection point.
Let $(\R^{2n},\omega_0)$ be the standard symplectic vector space and $l_1, l_2$ Lagrangian subspaces of $\R^{2n}$ which intersect transversally at the origin. 
Let $J:\R^{2n}\rightarrow \R^{2n}$ be an $\omega_0$-compatible almost complex structure with $l_2=J(l_1)$. 
Let $W=\{\xi\in l_1|\omega_0(\xi, J\xi)=1\}\cong S^{n-1}$. 
Define a map $F:W\times \R\rightarrow \R^{2n}, (\xi, t)\mapsto e^{-t} \xi+ e^t J\xi$. Then $F$ is a Lagrangian embedding, and $F$ is asymptotic to $l_1$ as $t\to -\infty$, and asymptotic to $l_2$ as $t\to +\infty$. One can smooth $F$  outside a large ball to obtain a Lagrangian embedding $F':W\times \R\rightarrow \R^{2n}$ such that $F'(W\times (-\infty, -c]) \subset l_1$ and $F'(W\times [c,+\infty))\subset l_2$ for some $c>0$. The image of $F'$, $\Gamma(l_1,l_2)$, is called a Lagrangian handle joining $l_1$ and $l_2$. 

In general, if $x$ is a transversal intersection point of two Lagrangian submanifolds $L_1, L_2$ or a transversal self intersection point of a Lagrangian submanifold $L$, by Weinstein neighborhood theorem, we can choose a neighborhood $U$ of $x$ such that $(L_1\cup L_2)\cap U$ or $L\cap U$ is the union of two Lagrangian disks. The Lagrangian surgery is cutting out $U$ and gluing back a portion of the Lagrangian handle carefully to construct a new Lagrangian submanifold. 

Notice that the Lagrangian isotopy class of a Lagrangian handle is independent of the choice of almost complex structures and smoothing. 
It turns out that the topological type of the resulting manifold is also completely determined. 
It will be much easier to describe it if we introduce the signs or orientations for the manifolds and the intersections. 
Assume the orientations of $l_1,l_2$ are given. 
The Lagrangian handle $\Gamma(l_1,l_2)$ is called positive or has sign 1 if the orientations of $l_1, l_2$ coincide in the image of $F'$. Otherwise, it is called negative or has sign $-1$. 
There is a natural orientation for the positive Lagrangian handle induced by the orientation of $l_1$ and $l_2$.

We know that
\begin{prop}\label{Polterovich}(\cite{P})
\begin{enumerate}
\item The sign of the Lagrangian handle $\Gamma (l_1, l_2)$ is $(-1)^{\frac{n(n-1)}{2}+1} \ind(l_1,l_2)$.
\item Let $P^n=S^{n-1}\times S^1$ and $Q^n=S^{n-1}\times [-1,1]/\sim$, where $(x, 1)\sim (\tau(x), -1)$, and $\tau:S^{n-1}\rightarrow S^{n-1}$ is an orientation reversing involution. Suppose $L$ is a connected immersed Lagrangian manifold with a self intersection point $x$ and $N$ is the resulting manifold of a Lagrangian surgery at $x$.
If $L$ is orientable, then $N\cong L\# P^n$ when the surgery is positive and $N\cong L\# Q^n$ when the surgery is negative.
If $L$ is non-orientable, then $N\cong L\# P^n\cong L\# Q^n$. 
\item 
\begin{enumerate}
\item $L, N$ are homologous (mod 2).
\item If $L$ is oriented and the surgery is positive, then $L, N$ are homologous. 
\end{enumerate}
\end{enumerate}
\end{prop}
When $n=2$, the sign of a Lagrangian handle coincides with $\ind(l_1,l_2)$. 
$N$ is diffeomorphic to $L\# T^2$ if $L$ is orientable and the index of self intersection point is 1. 
Otherwise, $N$ is diffeomorphic to $L\# KB$, where $KB$ denotes a Klein bottle. 


\subsubsection{Wavefront construction}
Let ${\bf x}=(x_1,\cdots ,x_n)$ be a coordinate system of $\R^n$ and ${\bf y}=(y_1,\cdots,y_n)$ the coordinates for fibers of $T^* \R^n$ corresponding to the basis $dx_1,\cdots, dx_n$.
Then $\lambda_{can}=\sum_iy_idx_i$ is the Liouville form on $T^* \R^n$, and 
$\omega_0=-d\lambda_{can}$ the standard symplectic structure.

A smooth section $f:L(\subset \R^n) \rightarrow (T^*\R^n,\omega_0)$ is Lagrangian if $f^*\omega_0=0$. So $f^*d\lambda_{can}=df^*\lambda_{can}=0$ and $f^*\lambda_{can}$ is closed. 
$f$ is called exact if $f^*\lambda_{can}=dh$ for some function $h$ on $L$. $h$ is called a generating function of $f$. The graph of $h$ in $L\times \R$ is called a wavefront of $f$.
Note that $h$ is unique up to a addition by a constant.

Conversely, let $h$ be a smooth function on $\R^n$. Consider the gradient function $\nabla h:\R^n\rightarrow \R^n$. The graph of $\nabla h$ is a Lagrangian section of $\R^{2n}\cong T^*\R^n$ with  a generating function $h$.

In general, if $L$ is a Lagrangian submanifold of a $2n$-dimensional symplectic manifold $(X, \omega)$ and $x\in L$, we can choose a neighborhood $U$ of $x$ and a local chart $\phi:U\rightarrow T^*\R^n$ carefully such that $\phi(L\cap U)$ is a Lagrangian section of $T^*\R^n$ with a generating function $h$.
If $f_1, f_2$ are two local sections with generating functions $h_1,h_2$ respectively, then $f_1, f_2$ coincide at $x\in \R^n$ if and only if $\nabla h_1(x)=\nabla h_2(x)$. 
Geometrically, it is equivalent to the condition that the wavefronts of $f_1,f_2$ are parallel at $x$. 
This approach is extremely useful for $n=2$ since we can explicitly draw the wavefronts, i.e. the graphs of $h_1, h_2$ in $\R^3$.


\begin{defn}
Let $L_1, L_2$ be two Lagrangian sections of $T^*\R^n$ which are given by $f_k:\R^n\rightarrow T^*\R^n$, i.e. $L_k=f_k(\R^n)$, and $h_1, h_2$ be their generating functions. 
Assume further that $\nabla h_1(x)=\nabla h_2(x)$ and $p=(x,h_1(x))$.
 We define $\sgn(p)=1$ if $\det(\Hess(h_2-h_1))>0$ and $\sgn(p)=-1$ if $\det(\Hess(h_2-h_1))<0$ at $p$.
\end{defn}

When $n=2$, $\sgn(p)$ can be visualized from the wavefronts easily.
We can assume further that $h_1(x)=h_2(x)$ by adding constants.
If $\det(\Hess(h_2-h_1))>0$, in a small neighborhood of $x$, one wavefront is completely in the top of the other one except $x$ (see Figure 1(a)).
When $\det(\Hess(h_2-h_1))<0$, these two wavefronts always intersect at some point near but not equal to $x$ (see Figure 1(b)). 

\centerline{
\includegraphics[width=1.0\linewidth]{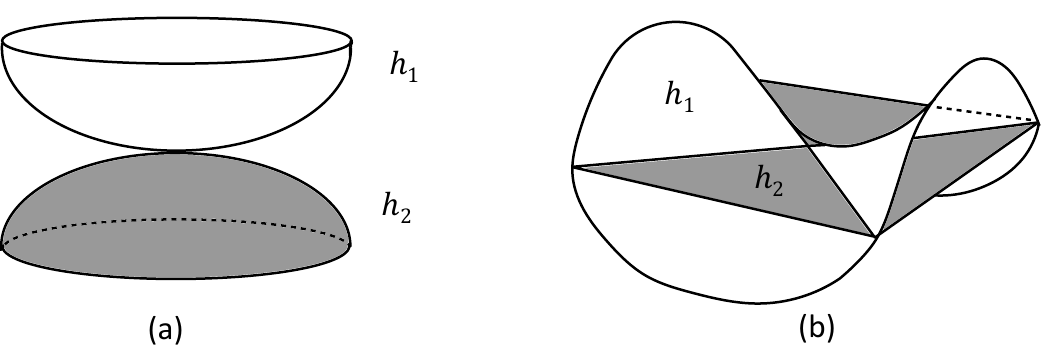}
}
{\centerline{Figure 1}}

\begin{defn}
Let $L_1, L_2$ be two oriented Lagrangian sections of $T^*\R^n$ which are given by $f_k:\R^n\rightarrow T^*\R^n$ and $h_1, h_2$ be their generating functions. 
Let $\R^n$ be oriented naturally by the ordered basis $\frac{\partial}{\partial x_1},\cdots ,\frac{\partial}{\partial x_n}$.
We define
\begin{enumerate}
\item 
$s(L_k)=1$ if $f_k$ is orientation-preserving, $s(L_k)=-1$ if $f_k$ is orientation-reversing. 
\item
$L_1, L_2$ have the same orientation or $s(L_1, L_2)=1$ if the map $f_2\circ f_1^{-1}:L_1\rightarrow L_2$ is orientation-preserving.
$L_1, L_2$ have opposite orientation or $s(L_1, L_2)=-1$ if $f_2\circ f_1^{-1}$ is orientation-reversing.
\end{enumerate}
It is clear that $s(L_1, L_2)=s(L_1)s(L_2)$.
\end{defn}

\begin{lem}\label{index}
Let $L_1, L_2$ be two oriented Lagrangian sections of $T^*\R^n$ intersecting transversally at $p=(x,y)$.
Let $h_1, h_2$ be their generating functions.
Let $\ind_p(L_1, L_2)$ denote the intersection index of $L_1, L_2$ at $p$.
Then 
 $$\ind_p(L_1, L_2)=(-1)^{\frac{n(n-1)}{2}}s(L_1, L_2)\, \sgn(p)$$
\end{lem}
\begin{proof}
The tangent plane $T_{(x,f_k(x))}L_k$ has a basis $e^k_i=\frac{\partial}{\partial x_i}+\sum_j\frac{\partial^2h_k}{\partial x_i\partial x_j}(x) \frac{\partial}{\partial y_j},$ $1\leq i\leq n$.

Since $L_1, L_2$ intersect transversally at $p=(x,y)$, then $y=f_1(x)=\nabla h_1(x)=\nabla h_2(x)=f_2(x)$ and $e_1^1,\cdots ,e_n^1, e_1^2,\cdots , e_n^2$ form a basis of $T_p(T^*\R^n)$.

The coefficient matrix of the ordered basis $e_1^1,\cdots ,e_n^1, e_1^2,\cdots , e_n^2$ with respect to the ordered basis $\frac{\partial}{\partial x_1},\cdots ,\frac{\partial}{\partial x_n},\frac{\partial}{\partial y_1}, \cdots \frac{\partial}{\partial y_n}$ is 
$$C=\begin{pmatrix} {\rm I\/} & \Hess(h_1) \\ {\rm I\/} & \Hess(h_2) \end{pmatrix}.$$
We can show that $\det(C)=\det(\Hess(h_2-h_1))$.
The sign between the basis $\frac{\partial}{\partial x_1},\cdots ,\frac{\partial}{\partial x_n},\frac{\partial}{\partial y_1},$ $\cdots \frac{\partial}{\partial y_n}$ and the standard basis $\frac{\partial}{\partial x_1}, \frac{\partial}{\partial y_1}, \cdots, \frac{\partial}{\partial x_n}, \frac{\partial}{\partial y_n}$ is $(-1)^{\frac{n(n-1)}{2}}$.
Hence 
$$\ind_p(L_1, L_2)=s(L_1)s(L_2) \,\sgn(\det(C))\cdot (-1)^{\frac{n(n-1)}{2}}=(-1)^{\frac{n(n-1)}{2}}s(L_1,L_2)\sgn(p)$$

\end{proof}

\begin{eg}\label{Givental}
\begin{enumerate}
\item Let $L$ be a constant section of $T^*\R^2$ given by $y_1=a, y_2=b$. $L$ is Lagrangian and a wavefront of $L$ is a plane $h(x_1,x_2)=ax_1+bx_2+c$ in $\R^3$.
The wavefronts of two different constant sections are two non-parallel planes.   
\item (Whitney sphere)
Let $S^n=\{({\bf x},y)\in \R^n\times \R| {\bf x}=(x_1,\cdots,x_n), ||{\bf x}||^2+y^2=1\}$ 
and $w:S^n\rightarrow \R^{2n}, w(x_1,\cdots ,x_n,y)=(x_1,\cdots,x_n,x_1y,\cdots, x_ny)$.

$w$ can be viewed as the union of two sections of $T^*\R^n\cong \R^{2n}$ over the unit disk:
$w_\pm({\bf x})=({\bf x},t_\pm({\bf x}){\bf x})$ where $t_\pm({\bf x})=\pm \sqrt{1-||{\bf x}||^2}$.
Let $L_\pm$ denote the graph of $w_\pm$.

$w_\pm$ have generating functions $h_\pm=\mp\frac13(1-||{\bf x}||^2)^\frac32$.
$\nabla h_+=\nabla h_-$ when $||{\bf x}||=1$ or ${\bf x}$ is the origin.
$L_\pm$ can be sewed smoothly along $||{\bf x}||=1$. 
So $w$ is an immersed Lagrangian $n$-sphere with a double point at the origin.

When $n=1$, $h_+$ and $h_-$ have same slopes at $x=0,1,-1$.
Moreover, if an orientation of $w(S^1)$ is given, it induces orientations for $L_\pm$.
For instance, if it is oriented as shown in Figure 2, then $L_+$ is oriented in $x_1$ direction and $L_-$ is oriented in $-x_1$ direction.
So $s(L_+,L_-)=-1$. 

\centerline{
\includegraphics[width=1.0\linewidth]{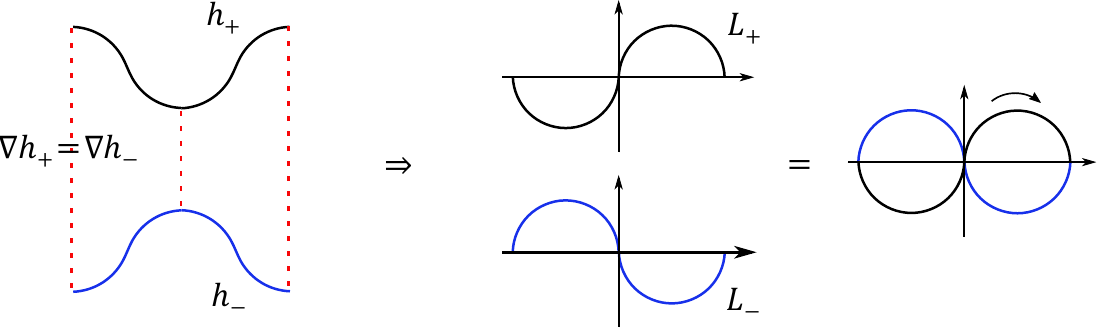}
}
{\centerline{Figure 2}}

When $n=2$, $h_+$ and $h_-$ have same slopes on the circle $||x||=1$ and at the origin $(0,0)$. 
Assume an orientation of $w(S^2)$ is given by a frame near the common boundary of $L_\pm$ as show in Figure 3.
When the frame is moved away from the boundary, we can keep the first vector in $x_2$ direction, then the second one is in $-x_1$ direction for $L_+$ but in $x_1$ direction for $L_-$.
So the induced orientations for $L_+, L_-$ are different (compare with the natural orientation of $\R^2$) and $s(L_+,L_-)=-1$. 

\centerline{
\includegraphics[width=1.0\linewidth]{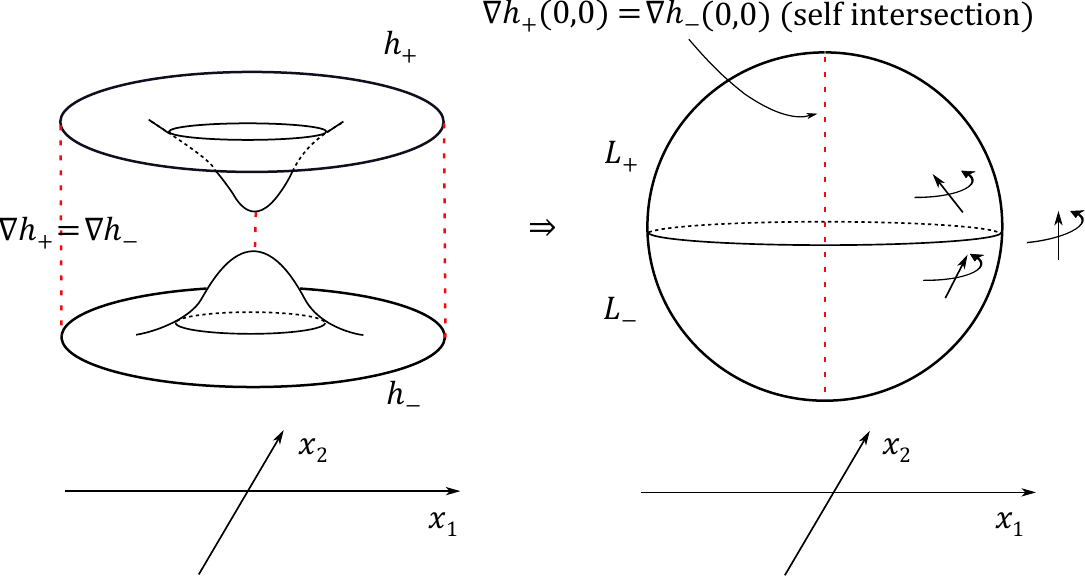}
}
{\centerline{Figure 3}}

Moreover, at the intersection point $p=(0,0,0,0)$, we know that $\Hess (h_+-h_-)=2{\rm I\/}, \det(\Hess (h_+-h_-))=4>0$ and $\sgn(p)=1$.
This also can be observed from Figure 1(a).

\item Let $L_1=L_+\cup L_-$ be the 2-dimensional Whitney sphere in (2) and $L_2$ be a nonzero constant section close to but not intersecting with $L_1$. 
Without loss of generality, we may assume $L_2$ has a generating function $h_2(x_1, x_2)=x_1$.
Let $h_-'$ be a generating function deformed from $h_-$ as shown in Figure 4 and $L_-'$ be its Lagrangian section.
$L_1'=L_+\cup L_-'$ is another Lagrangian sphere in $T^*\R^2$ with one double point $p=(0,0,0,0)$.

\centerline{
\includegraphics[width=0.6\linewidth]{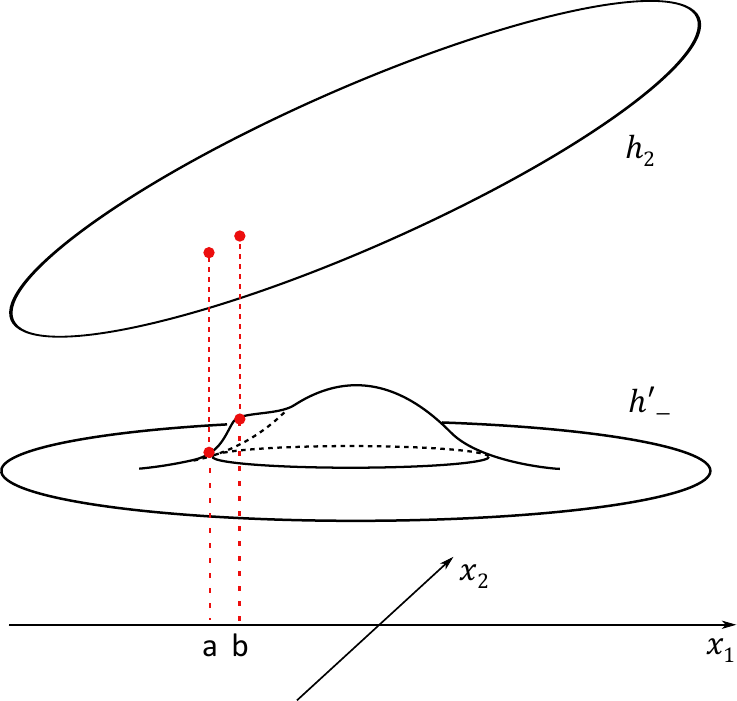}
}
{\centerline{Figure 4}}

Along the line $x_2=0$, $\nabla h_-'= \nabla h_2=(1,0)$ at two points $(a,0), (b,0), a<b<0$.
Moreover, $\frac{\partial h_-'}{\partial x_2}<0$ when $x_2>0$ and $\frac{\partial h_-'}{\partial x_2}>0$ when $x_2<0$ in the deformed area.
So $\nabla h_-'$ is never parallel to $\nabla h_2$ when $x_2\neq 0$ and  $L_-', L_2$ intersect transversally at two points $p_1=(a,0,1,0), p_2=(b,0,1,0)$.
Moreover, $\det(\Hess(h_-'-h_2))<0$ at $(a,0)$ and $\det(\Hess(h_-'-h_2))>0$ at $(b,0)$.
So $\sgn(p_1)=-1, \sgn(p_2)=1$.







\end{enumerate}
\end{eg}


\subsubsection{Surgeries on wavefronts}\label{surgery}
When we consider Lagrangian sections on the cotangent bundle $T^*\R^n$, the sign of Lagrangian handles can be easily read out from their generating functions.

\begin{lem}\label{sign}
Suppose two Lagrangian sections $L_1, L_2$ are oriented and intersect transversally at $p=(x,y)$.
Then the sign of the Lagrangian handle at $p$ is 
$$-s(L_1, L_2)\sgn(p).$$


\end{lem}
\begin{proof}
By Proposition~\ref{Polterovich}(1) and Lemma~\ref{index}, the sign of the Lagrangian handle at $p$ is 
$$(-1)^{\frac{n(n-1)}{2}+1}\cdot (-1)^{\frac{n(n-1)}{2}}s(L_1, L_2)\sgn(p)=-s(L_1, L_2)\sgn(p)$$
\end{proof}

As we mentioned in the paragraph before Proposition~\ref{Polterovich}, the topological feature of the resulting manifold after Lagrangian surgery is independent of the choice of the Lagrangian sections.
When $n=2$, the effect of the surgery is completely determined by $\sgn(p)$ and can be visualized from their wavefronts easily.

\centerline{
\includegraphics[width=0.8\linewidth]{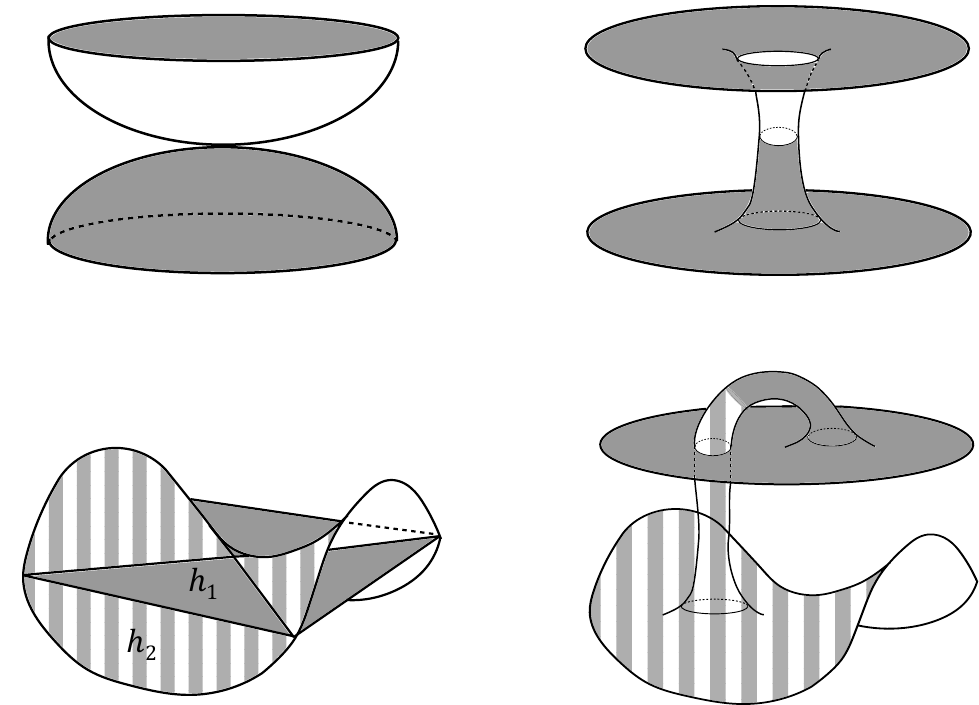}
}
{\centerline{Figure 5}}


\begin{eg} \label{eg25}
\begin{enumerate}
\item The 2-dimensional Whitney sphere in Example \ref{Givental}~(2) is an immersed sphere in $T^*\R^n$ with a self-intersection point at $p$. 
If it is oriented, we have shown that $s(L_+,L_-)=-1$ and $\sgn(p)=1$.
By Lemma~\ref{sign}, the Lagrangian handle at $p$ is positive and the  resulting manifold is an embedded Lagrangian torus in $T^*\R^2$.

\item Consider the Lagrangian surfaces $L_1', L_2$ in Example \ref{Givental}~(3).
There are three intersection or self intersection points: $p, p_1, p_2$.
Assume $L_1', L_2$ are oriented such that $s(L_-',L_2)=1$. 
By Lemma~\ref{sign}, the Lagrangian handle is positive at $p,p_1$ and negative at $p_2$.
If we apply Lagrangian surgery at $p_1$, the resulting manifold $L_3$ is diffeomorphic to $L_1'\# L_2$ with natural orientation and two double points $p, p_2$.
Applying the surgery to $L_3$ at $p$, by Proposition~\ref{Polterovich}, the resulting manifold $L_4$ also has natural orientation and is diffeomorphic to $L_1'\# L_2\# T^2$ with one double point $p_2$.
Finally, we can apply the surgery to $L_4$ at $p_2$ and get an embedded Lagrangian surface which is diffeomorphic to $KB\# L_2\# T^2$ or $L_2\# 4\R\P^2$.

\end{enumerate}
\end{eg}



\subsubsection{Existence of embedded non-orientable Lagrangian surfaces}
Now we can prove the existence of embedded non-orientable Lagrangian surfaces.
\begin{proof}[Proof of Proposition \ref{general existence}]
By a classical result of Thom (\cite{Thom54}), any mod 2 homology class is represented by a smooth map
 $f:\Sigma\to M$, for some  surface $\Sigma$.  
If necessary, by connecting sum with  an appropriate number of locally null-homologous $\R\P^2$, we can assume 
that $\Sigma$ is non-orientable and  $\chi(\Sigma)\equiv \langle w_2(X), A\rangle \pmod 2$.
Then there is a Lagrangian immersion $\Sigma\to (X, \omega)$ by Proposition \ref{h principle w2 version}. 
Perform Lagrangian surgeries on the double points to obtain an embedded non-orientable Lagrangian surface in the class $A$. 
\end{proof}

We can further construct more Lagrangian surfaces with arbitrary small Euler numbers.
 
\begin{prop} \label{add 4}
If $A$ is represented by a Lagrangian surface $L$, then it is represented by a Lagrangian surface of type $L\#(4l)\R\P^2$ for any positive integer $l$. 
\end{prop} 

\begin{proof}
Let $f:L\rightarrow X$ be a Lagrangian submanifold. By a local version of Weinstein's Lagrangian neighborhood theorem, for any $x\in L$, there exists a neighborhood $U$ and a symplectic embedding $\phi:(U,\omega)\to (T^*\R^n,\omega_0)$ such that $\phi(L\cap U)$ is the intersection of $\phi(U)$ with the constant section $L_2$ in Example~\ref{Givental}(3). 
As shown in Example \ref{eg25}~(3), we can take a null homologous immersed Lagrangian sphere with a self intersection point, and intersecting $L$ at two other points. Apply Lagrangian surgeries to construct a Lagrangian submanifold $L'$, which is diffeomorphic to $L\# 4\R\P^2$.
By Proposition~\ref{Polterovich}(3), $L'$ is in the same mod 2 homology class of $L$.

We can repeat this process and get $L\# (4l)\R\P^2$ for any $l\in \N$. 
\end{proof}

\begin{rem}
It is more natural to consider the Clifford torus to increase the genus.
Here we choose a deformation of Whitney sphere because it is easier to give a clear explanation for the whole process via wavefront in this case.
Actually, the construction in Example \ref{eg25}~(2) is generic and should work in more general situations. 
\end{rem}



\subsection{Lagrangian blow-up}
In this section, we describe the procedure of Lagrangian blow-up in \cite{Ri} and give an important example.  
Let $\tilde{B}=\{(z,l)\in \C^n\times \C\P^{n-1}\n z\in l\}$, $\tilde{B}_\varepsilon=\{(z,l)\in \ti B \n |z|\leq \varepsilon\}$, and $B_r=\{z\in \C^n \n |z| \leq r\}$. 
There are two natural projections $p_1: \ti B\to \C^n$, and $p_2: \ti B \to \C\P^{n-1}$. $p_1$ implies that $\ti B$ is the blowup of $\C^n$ at the origin, and $p_2$ implies that $\ti B$ is also the universal line bundle over $\C\P^{n-1}$.  For any $\l>0$, let $\omega_\lambda=p_1^*\omega_0+\lambda^2p_2^*\omega_{FS}$ be the induced symplectic form on $\tilde{B}$. 

There is a symplectomorphism $\alpha:(\tilde{B}_\varepsilon-\tilde{B}_0,\omega_\lambda)\cong (B_{\sqrt{\lambda^2+\varepsilon^2}}-B_\lambda,\omega_0)$, (cf. \cite{MS98}, Lemma 7.11).
In particular, it preserves the real parts of $\tilde{B}$ and $\C^n$. 

Suppose $X$ is a symplectic manifold and $x\in X$. 
By Darboux's theorem, there is a symplectic embedding $\psi: (B_\delta, \w_0) \to (X,\w)$ for $\delta\geq \sqrt{\lambda^2+\varepsilon^2}$ when $\l$ and $\varepsilon$ are sufficiently small, and $\psi(0)=x$. Let 
$\tilde{X}=(X-\psi(B_\lambda))\amalg \tilde{B}_{\varepsilon}/\sim$, with $\psi(\a(y)) \sim y$ for any $y\in \tilde{B}_\varepsilon-\tilde{B}_0$. 
The closed 2-form $\tilde{\omega}$ induced by $\omega$ and $\omega_0$ is a symplectic structure on $\tilde{X}$.

If $X$ is a symplectic 4-manifold, $L$ is a Lagrangian surface in $X$ and $x\in L$, then there exists a neighborhood $U$ of $x$ and a symplectic embedding $\phi:(U,\omega)\to (\C^{2},\omega_0)$ such that $\phi(L\cap U)$ is the real part.
By Theorem 1.21 in \cite{Ri}, after blowing up at $x$, $L$ is lifted to a Lagrangian surface $\tilde{L}\subset \tilde{X}$ with  $\tilde{L}\cong L\#\R\P^2$. Moreover, the mod $2$ class 
$[\tilde L]_2$ represented by $\tilde L$ satisfies 
$$[\tilde L]_2=[L]_2+E,$$ where $[L]_2$ is the mod $2$ class represented by $L$ and $E$ is the mod 2 reduction of the exceptional divisor, and we use the natural decomposition  $H_2(\tilde X, \Z_2)=H_2(X, \Z_2)\oplus \Z_2 E$. 

An important example is the holomorphic blow-up of $\C\P^2$:
$$\tilde{X}=\{ ([W_1:W_2],[Z_0:Z_1:Z_2]) \in \C\P^1 \times \C \P^2 \n Z_1 W_2 =Z_2 W_1 \}. $$

There are natural projections $p_i: \tilde{X}\to \C\P^i$, $i=1,2$. The projection $p_1: \tilde{X}\to \C\P^1$ is a nontrivial $\C\P^1$-bundle over $\C\P^1$. The preimage $p_2^{-1}([1:0:0]) =\C\P^1 \times \{[1:0:0]\}$ is the exceptional curve, and $p_2: \tilde{X}-p_2^{-1}([1:0:0]) \to \C\P^2 -\{[1:0:0]\}$ is a diffeomorphism. 

There is a family of K\"ahler forms on $\C\P^1 \times \C \P^2$, given by $\w=  p_1^*\tau_1  +\l^2 p_2^* \tau_2$, where $\l>0$, and $\tau_i$ are the Fubini-Study forms on $\C\P^i$. Then $\tilde{X}$ inherits a family of K\"ahler forms.

Let $\overline{H} \in H_2(\ti X;\Z)$ be the hyperplane class of $\C\P^2$, and $\overline E$ the exceptional class. 
The real locus of $\tilde{X}$ is diffeomorphic to $\R\P^2 \# \R\P^2 \cong KB$ representing $H+E \in H_2(\ti X,\Z_2)$. It is an embedded Lagrangian surface with respect to that family of  K\"ahler forms, and also has the structure of a nontrivial $\R\P^1$-bundle over $\R\P^1$. Note that a fiber of nontrivial $S^2$-bundle over $S^2$ represents the integral class $\overline{H}-\overline{E}$. Hence the Lagrangian KB and a fiber of the nontrivial $S^2$-bundle represent the same mod 2 class. 

\section{Minimal genus Lagrangian surfaces in rational  $4-$manifolds}

We will prove the following result 




\begin{prop} \label{prop28}
Let $X$ be a rational $4-$manifold 
and $A\in H_2(X;\Z_2)$ a nonzero class. $A$ is represented by a Lagrangian $l\R\P^2, 0\leq l\leq 3$ (for some symplectic structure) if and only if $\mathcal P(A)\equiv 2-l \pmod 4$.
Here, we use the convention $S^2=0 \R\P^2$. 
\end{prop}

\begin{proof}


The conditions are necessary by \cite{A1}. So we just need to show that they are sufficient.

First we consider the case $X=\C\P^2\# k\overline{\C\P^2}, k\in \N$.
Assume $\overline{H}$ is the generator of $H_2(\C\P^2,\Z)$, $\overline{E_i}$'s are the exceptional classes of $H_2(X,\Z)$ and $H, E_i\in H_2(X,\Z_2)$ are the reductions of $\overline{H}, \overline{E_i}$ respectively. Any class in $H_2(X, \Z_2)$ is of the form 
$$A=aH+b_1E_1+\cdots+ b_kE_k,$$
where $a, b_i$ are either $0$ or $1$. Note that $\mathcal{P}(E_1+\cdots+ E_k)=-k\mod4$ and $\mathcal{P}(H+E_1+\cdots+ E_k)=1-k\mod4$. 
All our conclusions and results are valid under permutation of exceptional divisors. For convenience, we will choose one class in discussion, which also works for any other classes of the same type.
For example, $H+E_1$ can also represent $H+E_2, H+E_3$ etc.  

Let $K_X=-3\overline{H}+\sum_i\overline{E_i}$ be the standard canonical class. 
For any $t\in\N$, let  
$$\overline{Z_t}=t\overline{H}-\overline{E_1}-\cdots-\overline{E_{2t+1}}-(t-1)\overline{E_{(2t+2)}}.$$ 
We explain that $\overline{Z_t}$ is represented by a smooth sphere.
Let $D$ be a configuration of $t$ degree one algebraic curves $D_1,\cdots ,D_{t}$ in $\C\P^2$, such that $D_1,\cdots, D_{t-1}$ pass through a common point $x$, while $D_t$ misses $x$. Blow up at $x$ and other $2t+1$ points in $D_t$ which are away from the intersection points $D_1\cap D_t,\cdots, D_{t-1}\cap D_t$, the lift $\tilde{D}$ of $D$ in $\C\P^2\# (2t+2)\overline{\C\P^2}$ is in the class $\overline{Z_t}$ if we arrange the indices of exceptional divisors appropriately. 
A sphere is given by resolving the intersecting points $D_1\cap D_t,\cdots, D_{t-1}\cap D_t$. 
Hence $\overline{Z_t}$ is represented by a smooth sphere.

By \cite{LW} Proposition 5.6, an integral class $Z$ is represented by a Lagrangian sphere (with respect to some symplectic structure $\omega$ with canonical class $K_X$) if and only if $Z$ is represented by a smooth sphere, $Z\cdot K_X=0$ and $Z^2=-2$. 
It can be shown straightforwardly that $\overline{Z_t}$ satisfies $\overline{Z_t}\cdot K_X=0$ and $\overline{Z_t}^2=-2$. 
Hence $\overline{Z_t}$ (and its reduction $Z_t$) is represented by a Lagrangian sphere.

The reduction of $\overline{Z_t}$ is $E_1+\cdots +E_{4l+2}$ when $t=2l$ and is   
$H+E_1+\cdots +E_{4l+3}$ when $t=2l+1$. Therefore, the  mod $2$ classes in the two sequences 
$$\{E_1+\cdots +E_{4l+2}, l\geq 0\}\quad \mbox{and}   \quad \{H+E_1+\cdots +E_{4l+3}, l\geq 0\}$$ are represented by Lagrangian spheres. 

We perform the Lagrangian blow-up construction (see Section 2.3) at one point of a Lagrangian sphere in the mod $2$ class $E_1+\cdots +E_{4l+2}$ to obtain a Lagrangian $\R\P^2$ in the mod $2$ class $$E_1+\cdots +E_{4l+3}$$
for any $l\geq 0$. And by repeating this process, for any $l\geq 0$, we obtain a  Lagrangian $2\R\P^2$ in the mod $2$ class $$E_1+\cdots +E_{4l+4}$$ and a Lagrangian $3\R\P^2$ in the mod $2$ class $$E_1+\cdots +E_{4l+5}.$$ Similarly, for any $l\geq 0$, by
blowing up at 1, 2 or 3 points of a Lagrangian sphere in the mod $2$ class  $H+E_1+\cdots +E_{4l+3}$, we can construct Lagrangian $\R\P^2, 2\R\P^2=KB$ or $3\R\P^2$ in the mod $2$ classes $$H+E_1+\cdots +E_{4l+4}, H+E_1+\cdots +E_{4l+5}, H+E_1+\cdots +E_{4l+6}$$ respectively.    

We are only left with the mod $2$ classes  $$0, H, E_1, H+E_1, H+E_1+E_2$$ to consider. 
The real part of $\C\P^2$ is a Lagrangian $\R\P^2$ of class $H$. A Lagrangian KB of $H+E_1$ and a Lagrangian $3\R\P^3$ of $H+E_1+E_2$ can be constructed by blowing up  one or two points on this $\R\P^2$. 
Blowing up at a point of the Clifford torus, we get a Lagrangian $3\R\P^2$ representing $E_1$. 

Since $(S^2 \times S^2)\# k\overline{\C\P^2}$ is diffeomorphic to $\C\P^2\# (k+1) \overline{\C\P^2}$ for any $k\geq 1$, the remaining rational $4$-manifold is $S^2 \times S^2$.  $H_2(S^2 \times S^2,\Z) \cong \Z \oplus \Z$ has two generators,  the base class (or section class) $\overline{B}=[S^2 \times pt]$ and the fiber class $\overline{F}=[pt \times S^2]$. 
Let $B, F\in H_2(S^2 \times S^2,\Z_2)$ denote their reductions.

Let $\phi: S^2\to S^2$ be the antipodal map, and $\w$ a symplectic form on $S^2$ such that $\phi^*\w=-\w$. Note that the standard symplectic form obeys this condition. Equip $S^2 \times S^2$ with the product symplectic form $\w \oplus \w$. It is easy to see that the graph of the antipodal map: $x\mapsto (x,\phi(x))$ is an embedded Lagrangian sphere representing the integral class $\overline{B}-\overline{F}$, hence the mod 2 class $B+F$.

Since the mod 2 classes $B$ and $F$ are symmetric, it suffices to construct Lagrangian surfaces for the class $F$.  We will construct such a Lagrangian representative by the real part of $\C\P^2\#  \overline{\C\P^2}$, and embed this Lagrangian surface into the symplectic fiber sum of two copies of $\C\P^2\# \overline{\C\P^2}$.

We recall the operation of symplectic fiber sum briefly (cf. \cite{Gom,MS98}). Let $(X_1,\w_1)$ and  $(X_2,\w_2)$ be symplectic manifolds of same dimension $2n$, and $(Q,\tau)$ be a compact symplectic manifold of dimension $2n-2$. Suppose that 
$$ \iota_i: Q\to X_i$$
are symplectic embeddings such that their images $\iota_i(Q)$ have trivial normal bundles. By the symplectic neighborhood theorem, there are symplectic embeddings 
$$ f_i: Q\times B^2 (\e)  \to X_i, \quad f_i^* \w_i =\tau \oplus dx\wedge dy,$$
such that $f_i(q,0) =\iota_i(q)$ for $q\in Q$, $i=1,2$. 

Let $A(\underline\e,\e)$ be the annulus on $B^2(\e)$ with radius $\underline\e <r<\e$, and $\phi: A(\underline\e,\e) \to A(\underline\e,\e)$ be an area- and orientation-preserving diffeomorphism which swaps the two boundary components. Then the symplectic fiber sum is defined by 
$$X_1 \#_Q X_2 = (X_1 - f_1 ( Q\times B^2(\underline{\e}) ) \bigcup (X_2-f_2 (Q\times B^2(\underline{\e})) /\sim,$$
where 
$$ f_2(q, z) \sim f_1 (q, \phi(z)), \quad \forall (q,z) \in Q\times A(\underline{\e},\e).$$
There is a natural symplectic structure on $X_1 \#_Q X_2$ induced by $\w_1$ and $\w_2$. 

Let us take two copies of $\C\P^2\# \overline{\C\P^2}$ as in the last part of Section 2, and denote them as $X_1$ and $X_2$. Let $L_1$ be the real locus of $X_1$ which is a Lagrangian Klein bottle. Regard $X_1$ and $X_2$ as nontrivial $S^2$-bundles over $S^2$, and perform symplectic fiber sum $X_1 \#_{S^2} X_2$, such that the gluing region in the base of $X_1$ is away from the real locus $\R\P^1 \subset \C\P^1$.  Denote the resulting manifold and symplectic form as $(\hat X, \hat\w)$. Then $\hat X$ is a trivial $S^2$-bundle over $S^2$, since $\pi_1(\text{Diff}^+(S^2))\cong \pi_1(SO(3))\cong \Z_2$. It is easy to see that $L_1$ is embedded in $\hat X$ as a  Lagrangian Klein bottle, representing the fiber class  in $H_2(\hat X,\Z_2)$.
\end{proof}



\ms
Let us complete the proof of Theorem \ref{main}. 

\begin{proof}[Proof of Theorem \ref{main}] 
It is given by Proposition~\ref{prop28} and Propsition~\ref{add 4}.
\end{proof}



\begin{thebibliography}{99}

\bibitem{A1} M. Audin,   \textit{Quelques remarques sur les surfaces lagrangiennes de Givental},  J. Geom. Phys. \textbf{7} (1990), no. 4, 583--598 (1991). 
\bibitem{ALP} M. Audin, F.LaLonde, L. Polterovich,   \textit{Symplectic rigidity: Lagrangian submanifolds}. Holomorphic curves in symplectic geometry, pp. 271--321, Progr. Math., \textbf{117}, BirkhŠuser, Basel, 1994.
\bibitem{B} P. Biran, \textit{Geometry of symplectic intersections}. Proceedings of the International Congress of Mathematicians, Vol. II (Beijing, 2002), 241--255, Higher Ed.Press, Beijing, 2002.
\bibitem{BLW} S. Borman, T. J. Li, W. Wu,  \textit{ Spherical Lagrangians via ball packings and symplectic cutting}. Selecta Math. (N.S.) \textbf{20} (2014), no. 1, 261--283. 

\bibitem{DLW}J. Dorfmeister, T. J. Li,  W. Wu,  \textit{Stability and existence of surfaces in symplectic 4-manifolds with $b^+=1$},  to appear in J. Reine Angew. Math.  
\bibitem{Gi} A. Givental, \textit{Lagrangian imbeddings of surfaces and unfolded Whitney umbrella}, Functional Anal. Appl. \textbf{20}:3 (1986), 197--203.
\bibitem{Gom} R. Gompf, \textit{A new construction of symplectic manifolds}, Ann. Math., 2nd Series, \textbf{142}:3 (1995), 527--595.
%
\bibitem{Gr} M. Gromov, \textit{Partial differential relations}, volume 9 of Ergebnisse der Mathematik und
ihrer Grenzgebiete (3) [Results in Mathematics and Related Areas (3)], Springer, Berlin
(1986).
\bibitem{LS} F. Lalonde, J.C. Sikorav, \textit{Sous-variétés lagrangiennes et lagrangiennes exactes des fibrés cotangents}. (French) Comment. Math. Helv. 66 (1991), no. 1, 18--33.
\bibitem{Lee} J. Lee,    \textit{On the classification of Lagrange immersions}, Duke Math. J. \textbf{43} (1976)
217--224. 
\bibitem{LW}T. J. Li, W. Wu,   \textit{Lagrangian spheres, symplectic surfaces and the symplectic mapping
  class group}.
 Geom. Topol., \textbf{16}:2 (2012), 1121--1169.

\bibitem{MS98} D. McDuff and D. Salamon, \textit{Introduction to Symplectic Topology, 2nd Edition}, Oxford Mathematical Monographs, Oxford Science Publications, 1998.  

\bibitem{MS74}  J.  Milnor and J.  Stasheff, \textit{Characteristic classes},  Princeton University
Press, Princeton, N. J.; University of Tokyo Press, Tokyo, 1974. Annals
of Mathematics Studies, No. 76.

\bibitem{N} S. Nemirovski, \textit{Homology class of a Lagrangian Klein bottle}, Izvestiya: Math. 
\textbf{73}:4 (2009), 689--698.
\bibitem{P} L. Polterovich, \textit{The surgery of Lagrange submanifolds}, Geom. Funct. Anal., \textbf{1} (1991), no. 2, 198--210.
\bibitem{Ri} A. Rieser, \textit{Lagrangian blow-ups, blow-downs, and applications to real packing}, J. Symplectic Geom., \textbf{12} (2014), no. 4, 725--789.
\bibitem{S} V. Shevchishin,   \textit{Lagrangian embeddings of the Klein bottle and combinatorial properties of mapping class groups}, Izvestia Math., \textbf{73}:4 (2009), 797--859.


\bibitem{Thom54} R. Thom, \textit{Quelques propri\'et\'es globales des vari\'et\'es diff\'erentiables}, Comment. Math. Helv. 28
(1954), 17-86.

\bibitem{T} E. Thomas,  \textit{A generalization of the Pontrjagin square cohomology operation}, Proc. Nat. Acad. Sci. U.S.A., \textbf{42} (1956), 266--269.

\end{thebibliography}
\end{document}